\documentclass[a4paper,oneside,11pt, reqno]{amsart}
\usepackage{eucal}
\usepackage{amsthm}
\usepackage{amsfonts}
\usepackage{amsmath}
\usepackage{amssymb}
\usepackage{mathrsfs}
\usepackage{epsfig}
\usepackage{cases}
\usepackage{tikz-cd}

\newtheorem{theorem}{\bf Theorem}[section]
\newtheorem{proposition}[theorem]{\bf Proposition}
\newtheorem{corollary}[theorem]{\bf Corollary}
\newtheorem{lemma}[theorem]{\bf Lemma}

\newtheorem{remark}[theorem]{\bf Remark}
\newtheorem{definition}[theorem]{\bf Definition}

\newcommand{\overbar}[1]{\mkern 1.5mu\overline{\mkern-1.5mu#1\mkern-1.5mu}\mkern 1.5mu}

\newenvironment{psmallmatrix}
 {\left(\begin{smallmatrix}}
 {\end{smallmatrix}\right)}

\newcommand{\z}{\mathbb Z_+}

\newcommand{\hl}{\mathcal{H}}
\newcommand{{\ran}}{\mbox{\rm ran}~}

\newcommand{\hla}{(\mathcal{H}, K^\alpha)\otimes(\mathcal{H}, K^\beta)}

\newcommand {\D} {\mathbb{D}}
\newcommand\quotient[2]{
        \mathchoice
            {
                \text{\raise1ex\hbox{$#1$}\Big/\lower1ex\hbox{$#2$}}%
            }{
                #1\,/\,#2
            }
            {
                #1\,/\,#2
            }
            {
                #1\,/\,#2
            }
    }

\usepackage{mathrsfs}

\begin{document}

\title[Gaussian curvature and the curvature of a Cowen-Douglas operator]{The relationship of the Gaussian curvature with the curvature of a Cowen-Douglas operator}
\author[S. Ghara]{Soumitra Ghara}
\address[]{}
\email[S. Ghara]{ghara90@gmail.com}
\author[G. Misra]{Gadadhar Misra}
\address[G. Misra]{Statistics and Mathematics Unit, Indian Statistical Institute, Bangalore 560059, and Department of Mathematics, Indian Institute of Technology, Gandhinagar 382055}
\email[G. Misra]{gm@isibang.ac.in}

\thanks{Support for the work of S. Ghara was provided by 
SPM Fellowship of the CSIR and a post-doctoral Fellowship of the Fields Institute for Research in Mathematical Sciences, Canada.
Support for the work of G. Misra was provided in the 
form of the J C Bose National Fellowship, Science and Engineering Research Board.  Some of the results in this paper are
from the PhD thesis of the first named author submitted to the Indian Institute of Science.}


%
%
%
\begin{abstract}
It has been recently shown that if $K$ is a sesqui-analytic scalar valued non-negative definite kernel on a  domain $\Omega$ in $\mathbb C^m$, then the function $\big(K^2\partial_i\bar{\partial}_j\log K\big )_{i,j=1}^ m,$ 
is also  a non-negative definite kernel on  $\Omega$. In this paper, we discuss two consequences of this result. The first one strengthens the curvature inequality for operators in the Cowen-Douglas class $B_1(\Omega)$ while the second one  gives a relationship of the reproducing kernel of a submodule of certain Hilbert modules with the curvature of the associated quotient module.
\end{abstract}

 \maketitle
\section{Introduction}
Let $X$ be an arbitrary set and let $K:X\times X\to{\mathcal M}_n(\mathbb{C})$, $n\geq 1$, be a function. We say that $K$  is  a non-negative definite kernel (resp. positive definite kernel) if for any subset $\{x_1,\ldots,x_p\}$ of $X$, 
the $ np \times np $ matrix $ \Big (\!\!\Big(\,K(x_i,x_j)\,\Big)\!\!\Big)_{i,j=1}^{p}$ 
is non-negative definite (resp. positive definite).
A Hilbert space $\mathcal H$ consisting of functions on $X$ is said to be a reproducing kernel Hilbert space with reproducing kernel $K$ if 
\begin{itemize}
\item[\rm (i)]
for each $x \in X $ and $\eta \in {\mathbb{C}}^n$, $ K(\cdot,x)\eta \in \mathcal H $
\item[\rm (ii)] for each $f\in \mathcal{H}$ and $x\in X $, ${\langle f,K(\cdot,x)\eta\rangle}_\mathcal{H}={\langle f(x),\eta\rangle}_{{\mathbb{C}}^n}.$ 
\end{itemize}
The  kernel $K$ of a reproducing kernel Hilbert space $\mathcal H$ is non-negative definite. Conversely, corresponding to each non-negative definite kernel $K$ there exists a unique  reproducing kernel Hilbert space $(\mathcal H, K)$ whose reproducing kernel is $K$ (see \cite{Aro},  \cite{PaulsenRaghupati}). 
For $K:X \times X\to \mathcal M_n(\mathbb{C})$, we write $K\succeq 0$ to denote that $K$ is non-negative definite. Analogously, we write $K\preceq 0$ if $-K$ is non-negative definite. For $K_1, K_2:X \times X\to \mathcal M_n(\mathbb{C})$, we write $K_1\succeq K_2$ to denote that $K_1-K_2\succeq 0$ and we write  $K_1\preceq K_2 $  if $K_1-K_2\preceq 0$. For any domain $\Omega$ in $\mathbb C^m$, $m\geq 1$, a function $K:\Omega\times \Omega\to{\mathcal M}_n(\mathbb{C})$ is said to be sesqui-analytic if it is holomorphic in first $m$-variables and anti-holomorphic in the second set of $m$-variables. In this paper, we will deal with non-negative definite kernels which are sesqui-analytic.

We now discuss an important class of operators introduced by Cowen and Douglas (see \cite{CD}, \cite{Curtosalinas}). 
Let 
$\boldsymbol T:=(T_1,...,T_m)$ be a $m$-tuple of commuting bounded linear operators on a separable Hilbert space 
$\mathcal H.$ Let $D_{\boldsymbol T}:\mathcal H\to \mathcal H\oplus\cdots\oplus \mathcal H$
be the operator defined by $D_{\boldsymbol T}(x)=(T_1x,...,T_mx), ~x\in \mathcal H.$ 
\begin{definition}[Cowen-Douglas class operator]
Let $\Omega\subset \mathbb C^m$ be a bounded domain. A commuting $m$-tuple $\boldsymbol T$ on $\mathcal H$ is said to be in the Cowen-Douglas class $B_n(\Omega)$ if $\boldsymbol T$ satisfies the following requirements:
\begin{enumerate}
\item[\rm(i)] 
\rm dim $\ker D_{\boldsymbol T-w} = n,~~w\in\Omega$
\item[\rm(ii)]
$\ran D_{\boldsymbol T-w}$ is closed for all $w\in\Omega$
\item[\rm(iii)]
$\overbar \bigvee \big\{\ker D_{\boldsymbol T-w}: w\in \Omega \big\}=\mathcal H.$ 
\end{enumerate} 
\end{definition}
If $\boldsymbol T\in B_n(\Omega)$, then  for each $w\in \Omega$, there exist functions $\gamma_1,\ldots,\gamma_n$ holomorphic in a neighbourhood $\Omega_0\subseteq \Omega$ containing $w$ such that $\ker D_{\boldsymbol T -w^\prime}=\bigvee\{\gamma_1(w^\prime),\ldots,\gamma_n(w^\prime)\}$ for all $w^\prime \in \Omega_0$ (cf. \cite{CDopen}). Consequently, every $\boldsymbol T\in B_n(\Omega)$ corresponds to a rank $n$
holomorphic hermitian vector bundle $E_{\boldsymbol T}$ defined by 
$$E_{\boldsymbol T}=\{(w,x)\in \Omega\times \mathcal H:x\in \ker D_{\boldsymbol T-w}\}$$ 
and $\pi(w,x)=w$, $(w,x)\in E_{\boldsymbol T}$.
For a bounded domain $\Omega$ in $\mathbb C^m$, let $\Omega^*=\{z:\bar{z}\in \Omega\}.$
It is known that if $T$ is an operator in $B_n(\Omega^*)$,
then  for each $w\in \Omega$, $T$ is unitarily equivalent to the adjoint of the multiplication tuple $M=(M_1,\ldots,M_m)$ on some 
reproducing kernel Hilbert space $(\mathcal H, K)\subseteq {\rm Hol}(\Omega_0, \mathbb C^n)$  for some  open subset  $\Omega_0\subseteq\Omega$ containing $w$.
If $T\in B_1(\Omega^*)$, the curvature matrix $\mathcal K_T(\bar{w})$ at a fixed but arbitrary point $\bar{w}\in \Omega^* $ is defined by
$$\mathcal K_T(\bar{w})=-\Big (\!\!\Big(\,\partial_i\bar{\partial}_j \log \|\gamma(\bar{w})\|^2\,\Big )\!\!\Big)_{i,j=1}^m,$$
where $\gamma$ is a holomorphic frame of $E_{T}$ defined on some open subset $\Omega_0^*\subseteq \Omega^*$ containing $\bar{w}$. Here,  $\partial_i$ and $\bar{\partial}_j$ denote $\frac{\partial}{\partial w_i}$ and $\frac{\partial}{\partial \bar{w}_j}$, respectively. 
If $T$
is realized as the adjoint of the multiplication tuple $M$ on some reproducing kernel Hilbert space $(\hl, K)\subseteq \rm{Hol}(\Omega_0)$, where $w\in \Omega_0$,  the curvature $\mathcal K_T(\bar{w})$  is then equal to $$-\Big (\!\!\Big(\,\partial_i\bar{\partial}_j \log K(w,w)\,\Big )\!\!\Big)_{i,j=1}^m.$$ 

Let $\Omega\subset \mathbb C$ be open and $\rho:\Omega \to \mathbb R_+$ be a $C^2$-smooth function. The Gaussian curvature of the metric 
$\rho$ is given by  the formula  
\begin{equation}\label{Gauss}
\mathcal G_{\rho}(z)=-\frac{\big (\partial\bar{\partial} \log \rho\big) (z)}{\rho(z)^2}, z\in \Omega.\end{equation}
If $K:\Omega\times \Omega\to \mathbb C$
is a non-negative definite kernel with $K(z,z)>0$, then the function $\frac{1}{K}$
defines a metric on $\Omega$ and its Gaussian curvature is given by the formula 
$$\mathcal G_{K^{-1}}(z)=K(z,z)^2  \Big(\,\partial\bar{\partial} \log K\,\Big)(z,z), \,\, z \in \Omega.$$
Since $\mathcal G_{K^{-1}}(z)$ can also be written as 
$K(z,z)\partial\bar{\partial}K(z,z)-\partial K(z,z)\bar{\partial}K(z,z)$, it follows that $\mathcal G_{K^{-1}}(z)$ can be extended to a sesqui-analytic function ${\mathcal G}_{K^{-1}}(z,w)$ on $\Omega\times \Omega$. It is therefore natural to extend the definition of the Gaussian curvature to an open subset $\Omega \subset\mathbb C^m$. Thus, for any non-negative definite kernel $K$ on $\Omega$,  we define
\begin{equation}
{\mathcal G}_{K^{-1}}(z,w):= \Big (\!\!\Big(\, K(z,w)\partial_i 
\bar{\partial}_j K(z,w)-\partial_i K(z,w)\bar{\partial}_j K(z,w)\Big)\!\!\Big)_{i,j=1}^m,\;\;z,w\in \Omega,
\end{equation}
where, with a slight abuse of notation, we let the symbols $\partial_i$ and $\bar{\partial}_j$ also stand $\frac{\partial}{\partial z_i}$ and $\frac{\partial}{\partial \bar{w}_j}$, respectively. 

\begin{proposition}\label{cork^2curv}\rm(\cite[Proposition 2.3]{GM})
Let $\Omega\subset\mathbb{C}^m$  be a domain and    
$K:\Omega \times \Omega \to \mathbb C $ be a sesqui-analytic function. Let $\alpha,\beta$ be two positive real numbers. Suppose that  $K^\alpha$ and $K^\beta$, defined on $\Omega\times \Omega,$ are non-negative definite for some $\alpha,\, \beta >0$. 
Then the function 
$\mathbb K^{(\alpha, \beta)}:\Omega \times \Omega \to \mathcal M_m(\mathbb C)$ defined by  
$$\mathbb K^{(\alpha, \beta)}(z,w):= K^{\alpha+\beta}(z,w)\Big (\!\!\Big(\,\big(\partial_i\bar{\partial}_j\log K\big)(z,w) \,\Big )\!\!\Big)_{i,j=1}^ m, \,\, z,w \in \Omega,$$ 	
is a non-negative definite kernel on $\Omega\times\Omega$ taking values in $\mathcal M_m(\mathbb C)$.
\end{proposition} 
We obtain the following corollary, saying that  ${\mathcal G}_{K^{-1}}(z,w)$ is a non-negative definite kernel whenever $K$ is non-negative definite, by setting $\alpha=1=\beta$. 
\begin{corollary}\label{gaussiancurv}
Let $\Omega$ be a  domain in $\mathbb{C}^m$. Suppose that 
 $K:\Omega\times\Omega\to\mathbb C$ is a sesqui-analytic non-negative definite kernel. Then 
${\mathcal G}_{K^{-1}}$
is also a non-negative definite kernel on $\Omega$ taking values in $\mathcal M_m(\mathbb C)$.
\end{corollary}

The introduction of the Gaussian curvature has many advantages and 
Corollary \ref{gaussiancurv} serves as a handy tool for many proofs. This is already apparent from \cite{infinitelydivisiblemetric}, many more examples are given in Section 2 of this paper. We have attempted  to strengthen the curvature inequality in the hope of obtaining a criterion for contractivity of operators in $B_1(\mathbb D)$. We haven't succeeded in doing this yet but several partial answers  that we have obtained indicate that one of these inequalities may do the job. In Section 2, we establish a monotonicity property of the Gaussian curvature. We conclude Section 2 by showing that the partial derivatives from $(\mathcal H, K)$ to $(\mathcal H, \mathcal G_{K^{-1}})$ are bounded. In the third Section we discuss the decomposition of the tensor product of two Hilbert modules, say $\mathcal M_1 \subset \text{\rm Hol}(\Omega)$ and $\mathcal M_2 \subset \text{\rm Hol}(\Omega)$. The tensor product  
$\mathcal M_1\otimes \mathcal M_2$ consists of holomorphic functions on $\Omega \times \Omega$. We consider the nested set of submodules $\mathcal M_1\otimes \mathcal M_2 \supset \mathcal A_0 \supset \mathcal A_1 \supset \cdots \supset \mathcal A_k \supset \cdots$, where $\mathcal A_k$  is the submodule of functions in $\mathcal M_1\otimes \mathcal M_2$ vanishing on the diagonal subset $\Delta$ of $\Omega\times\Omega$ along with their derivatives to order $k$. Setting $\mathcal S_k:= \mathcal A_{k-1}\ominus \mathcal A_{k}$, we have the direct sum decomposition 
$$\mathcal M_1\otimes \mathcal M_2= \bigoplus_{k=1}^\infty \mathcal S_k,$$ 
which one may think of as the Clebsch-Gordon decomposition for Hilbert modules.  
We also have the short exact sequence of Hilbert modules: 
$\begin{tikzcd}
0\arrow{r} &\mathcal A_0 
\arrow{r} {i} & {\mathcal M_1\otimes \mathcal M_2}
\arrow{r}{\pi} & \mathcal S_0
 \arrow{r}& 0.
\end{tikzcd}$
It is important to be able to find invariants for $\mathcal S_0$ from the inclusion
$\mathcal A_0 \subset \mathcal M_1\otimes \mathcal M_2$. In Section 3, in a large class of examples, we find such an invariant, see Theorem \ref{limitq} and the Remark following it. 

\section{Remarks on Curvature inequality}
In this section, we will discuss the curvature inequality for a contractive operator $T:\mathcal H \to \mathcal H$ in the Cowen-Douglas class  $B_1(\mathbb D)$ taking into account Corollary \ref{gaussiancurv}.  First, since the operator $T \in B_1(\mathbb D)$, it follows that the map $\gamma_T:\Omega \to Gr(\mathcal H,1)$, $\gamma_T(w) = \ker (T-w)$, $w\in \mathbb D$, is holomorphic. Here,  $Gr(\mathcal H,1)$ is the Grasmannian of $\mathcal H$ consisting of the $1$ dimensional subspaces. 
Second, any operator $T$ in $B_1(\mathbb D)$ is unitarily equivalent to  the adjoint $M^*$ of the operator $M$ of multiplication by the coordinate function $z$ on some reproducing kernel Hilbert space $(\mathcal H, K)\subseteq {\rm Hol}(\mathbb D)$. 
In particular, any  contraction $T$ in $B_1(\mathbb D)$, modulo unitary equivalence, is of this form.  Also, $M^* K(\cdot, w) = \overbar{w} K(\cdot,w)$, therefore we can take the map $\gamma_T(\overbar{w}) = \mathbb C [K(\cdot , w)]$ and with a slight abuse of notation, we shall write $\gamma_T(\overbar{w}) = K(\cdot , w)$. It is then easy to verify that $(M^* - \overbar{w}I) \overbar{\partial} K(\cdot ,w)=K(\cdot , w)$. Consequently setting $\mathcal N(w)$ to be the $2$ dimensional space $\{K(\cdot , w), \overbar{\partial} K(\cdot , w)\}$, we have that $(M-wI)^*_{|\mathcal N(w)} = \Big (\begin{smallmatrix}0 & 1 \\ 0 & 0 \end{smallmatrix} \Big )$. However if we represent 
$(M-wI)^*_{|\mathcal N(w)}$ with respect to the orthonormal basis $e_1(w), e_2(w)$ obtained by applying the Gram-Schmidt process to the pair of vectors $K(\cdot,w), \overbar{\partial} K(\cdot, w)$, then we have the representation: 
$$N_T(w): = (M-wI)^*_{|\mathcal N(w)}=
\begin{pmatrix}
0& (-\mathcal K_{T}(\overbar{w}))^{-\tfrac{1}{2}}\\
0&0
\end{pmatrix},~w\in\mathbb D.$$

The contractivity of the operator $M$, or equivalently, that of $M^*$ implies that 
the local operators $N_T(w) + \overbar{w} I$, $w\in \mathbb D$, must be contractive. Since a $2\times 2$ matrix of the form $\begin{psmallmatrix} w & \lambda\\0 &  w\end{psmallmatrix}$ is contractive if and only if $|\lambda| \leq 1 - |w|^2$, we 
obtain the curvature inequality of \cite{GMCI} reproduced in the form of a proposition below. 

\begin{proposition} \label{cid}
If $T$ is contraction in $B_1(\mathbb D)$, then the curvature of $T$ is bounded above by the curvature of the backward shift operator  $S^*$. 
\end{proposition} 
Without loss of generality, we may assume that the operator $T$ has been relaized as the adjoint of the multiplication operator $M$ on some Hilbert space of holomorphic functions $(\mathcal H, K)$. Note that  $-\mathcal K_T(w)= \partial\overbar{\partial} \log K(w,w)$ and  
the curvature $\mathcal K_{S^*}(w)$ of the backward shift operator $S^*$ is $-\partial\bar{\partial}\log  \mathbb S_{\D}(z,z)$, where $\mathbb S_{\D}(z,w)=\frac{1}{1-z\overbar{w}}$ is the Szeg\"o keenel of the unit disc. In otherwords, 
for a contractive operator $M^*$ in $B_1(\mathbb D)$, the curvature inequality  takes the form (see  \cite{infinitelydivisiblemetric}): 
\begin{equation} \label{eqn:curvinq}
-\partial\bar{\partial}\log K(z,z)\leq -\partial\bar{\partial}\log \mathbb S_{\D}(z,z)=- \tfrac{1}{(1-|z|^2)^2} ,~z\in \mathbb D.
\end{equation}
From the discussion preceding Proposition \ref{cid}, it is clear that the curvature inequality of a contractive operator in $B_1(\mathbb D)$ is nothing but the contractivity of its restriction to  the $2$ dimensional subspaces $\mathcal N(w)$, $w\in \mathbb D$. So, it is clear that the curvature inequality, in general, is not enough to ensure contractivity. 
We reproduce an example from \cite{infinitelydivisiblemetric} illustrating this phenomenon. 

Let $K_0(z,w)=\frac{8+8z\bar{w}-(z\bar{w})^2}{1-z\bar{w}}$, $z,w\in\mathbb D.$
Note that $K_0(z,w)$ can be written in the form 
$8+16z\bar{w}+15\frac{(z\bar{w})^2}{1-z\bar{w}},$ therefore it defines a non-negative definite kernel on the unit disc.
It is  not hard to see that, in this case 
$$\mathcal K_{M^*}(w)-\mathcal K_{S^*}(w)=-\frac{8(8-4|w|^2-|w|^4)}{1-|w|^2}\leqslant 0,~w\in\mathbb D.$$
Recall that for any reproducing kernel Hilbert space $(\mathcal H, K)$, the operator $M^*$ on $(\hl ,K)$ is a contraction if and only  if that the function $G(z,w):=(1-z\bar{w})K(z,w)$ is non-negative definite on $\mathbb D\times\mathbb D$ (see \cite[Corollary 2.37]{pick-int}). Since $(1-z\bar{w})K_0(z,w)=8+8z\bar{w}-(z\bar{w})^2$ which is not a non-negative definite kernel on the unit disc, it follows that the operator $M^*$ on $(\mathcal H, K_0)$ is not a contraction.

Since the curvature is a complete unitary invariant in the class $B_1(\mathbb D)$, one attempts to strengthen the curvature inequality in the hope of finding a criterion for contractivity in terms of the curvature. One such possibility is discussed in the paper \cite{infinitelydivisiblemetric}  replacing the point-wise inequality of \eqref{eqn:curvinq} by requiring that $0 \preceq \partial\bar{\partial}\log K(z,w)  -\partial\bar{\partial}\log  \mathbb S_{\D}(z,w)$, that is, 
$$\Big (\!\! \Big (\partial\bar{\partial}\log K(w_i,w_j)  - \partial\bar{\partial}\log  \mathbb S_{\D}(w_i,w_j)\Big )\!\!\Big )_{i,j=1}^n$$ 
is non-negative definite for all finite subsets $\{w_1, \ldots ,w_n\}$ of $\mathbb D$  and $n\in \mathbb N$. Here, we discuss a different strengthening of the curvature inequality \eqref{eqn:curvinq}. 

\begin{proposition}\label{eqn:curvinqstrong}
Let $T\in B_1(\mathbb D)$ be a contraction. Assume that 
$T$ is unitarily equivalent to the operator $M^*$ on $(\mathcal H, K)$ for some non-negative definite kernel $K$
on the unit disc. Then the following inequality holds:
\begin{equation} \label{eqn:4}
K^2(z,w)\preceq \mathbb S^{-2}_{\D}(z,w)\mathcal G_{K^{-1}}(z,w),
\end{equation}
that is, the matrix
\begin{equation*}
\Big (\!\! \Big ( \mathbb S^{-2}_{\D}(w_i,w_j)\mathcal G_{K^{-1}}(w_i,w_j)-K^2(w_i,w_j)\Big )\!\!\Big )_{i,j=1}^n
\end{equation*}
is non-negative definite for every subset $\{w_1,\ldots,w_n\}$ of $\mathbb D$ and $n\in \mathbb N$.
\end{proposition}
\begin{proof}
Setting $G(z,w)=(1-z\bar{w})K(z,w)$, we see that
\begin{align*}
& -G(z,w)^2\partial\bar{\partial}\log G(z,w)\\
& \quad\quad\quad  = (1-z\bar{w})^2K^2(z,w) \big(-\partial\bar{\partial}\log K(z,w) + \partial\bar{\partial}\log \mathbb S_{\D}(z,w)\big),~z,w\in \mathbb D.
\end{align*}
Therefore, since $G(z,w)$ is non-negative definite on $\mathbb D\times\mathbb D$, applying Corollary $\ref{gaussiancurv}$ for $G(z,w)$, we obtain that 
\begin{equation*} 
(1-z\bar{w})^2K(z,w)^2\big(-\partial\bar{\partial}\log K(z,w) + \partial\bar{\partial}\log \mathbb S_{\D}(z,w)\big)\preceq 0.
\end{equation*}
Since $\mathbb S_{\D}(z,w)^{-2}\partial\bar{\partial}\log \mathbb S_{\D}(z,w)=1$, the proof is complete.
\end{proof}
In particular, evaluating \eqref{eqn:4} at a fixed but arbitrary point, the inequality \eqref{eqn:curvinq} is evident.  However, for any contraction $T$ in $B_1(\mathbb D)$ (realized as $M^*$ on $(\mathcal H, K)$), the inequality \eqref{eqn:4} gives a much stronger (curvature) inequality as shown in the computation given below.  Conversely, whether it is strong enough to force contractivity of the operator $M^*$  is not clear. For a different approach, see \cite{Kai}.

In order to show that the inequality \eqref{eqn:4} is stronger than the inequality \eqref{eqn:curvinq}, it suffices to prove the kernel $K_0$ does not satisfy \eqref{eqn:4}. Setting $G_0(z,w)=(1-z\bar{w})K_0(z,w)$, we get 
$G_0(z,w)=8+8z\bar{w}-(z\bar{w})^2$, $z,w\in \mathbb D$. Thus
\begin{align*}
G_0(z,w)^2\partial\bar{\partial}\log G_0(z,w)
&=G_0(z,w)\partial\bar{\partial}G_0(z,w)-\partial G_0(z,w)\bar{\partial}G_0(z,w)\\
&=(8+8z\bar{w}-(z\bar{w})^2)(8-4z\bar{w})-(8z-2z^2\bar{w})(8\bar{w}-2z\bar{w}^2)\\
&=64-32z\bar{w}-8(z\bar{w})^2,
\end{align*}
which is clearly not a non-negative definite kernel. Hence
the operator $M^*$ on $(\mathcal H,K_0)$ does not satisfy  inequality \eqref{eqn:4}.

\begin{remark}We now have the following remarks.
\begin{enumerate} \item[(i)]  Under the assumptions of Proposition \ref{eqn:curvinqstrong}, it follows from \cite[Theorem 5.1]{PaulsenRaghupati} that the Hilbert space 
$(\mathcal H,K^2)$ is contained in the Hilbert space $(\mathcal H, \mathbb S^{-2}_{\D}\mathcal G_{K^{-1}})$, and  
the inclusion map from $(\mathcal H,K^2)$ to $(\mathcal H, \mathbb S^{-2}_{\D}\mathcal G_{K^{-1}})$ is contractive.
\item[(ii)] Recall that unitary equivalence class of the operator $M$ acting on a reproducing kernel Hilbert space $(\mathcal H, K)$ is determined by the kernel $K$ modulo pre- and post-multiplcation by a non-vanishing holomorphic function and its conjugate, see \cite[Theorem 3.7]{Curtosalinas} and the remark following it. The Guassian curvature $\mathcal G_{K^{-1}}$ of a non-negative definite kernel $K$ clearly depends on the  choice of the kernel $K$ and therefore is not a function of the unitary equivalence class of the operator $M$. However, we note that the validity of the inequality \eqref{eqn:4} depends only on the unitary equivalence class of the operator $M$.  
\end{enumerate}
\end{remark}

Let $\Omega$ be a finitely connected bounded planar domain and ${\rm Rat}(\Omega^*)$ be the ring of rational functions with poles off  $\overline{\Omega^*}$. Let $T$ be an operator in  $B_1(\Omega^*)$ with $\sigma(T)=\overline{\Omega^*}$. Suppose that the homomorphism $q_T:{\rm Rat}(\Omega^*) \to  B(\mathcal H)$ given by 
$$q_T(f) = f(T),\;\;f\in {\rm Rat}(\Omega^*),$$ is contractive, that is, $\|f(T)\| \leq \|f\|_{\Omega^*, \infty}$, $f\in {\rm Rat}(\Omega^*)$. As before, we think of $T$ as the adjoint $M^*$ of the multiplication operator $M$ on some reproducing kernel Hilbert space $(\mathcal H, K) \subset \text{\rm Hol}(\Omega)$.  Setting $G_f(z,w) = (1- f(z)\overbar{f(w)}) K(z,w)$ and using the 
the contractivity of $f(M^*)$, $\|f\|_{\infty, \Omega} \leq 1$, we have that $G_f\succeq 0$. Applying Corollary \ref{gaussiancurv}, we conclude that  
\begin{align*}
\begin{split}
0\preceq & G_f(z,w)^2 \partial \bar{\partial} \log G_f(z,w) \\
= &  G_f(z,w)^2 \Big (- \tfrac{f^\prime(z)\overbar{f^\prime(w)}}{(1 - f(z) \overbar{f(w)})^2} + \partial \bar{\partial} \log K (z,w) \Big )\\
= & - K(z,w)^2 f^\prime(z)\overbar{f^\prime(w)} + (1 - f(z)\overbar{f(w)})^2 K(z,w)^2 \partial \bar{\partial} \log K (z,w)
\end{split}
\end{align*}
for any rational function $f$  with poles off~ $\overbar \Omega$ and 
$| f(z) | \leq 1,$ $z\in \Omega$.  Also, if $f^\prime$ is a non-vanishing function on $\Omega$, then the pull-back of the metric induced by the Szeg\"{o} kernel is the metric $f^*(\mathbb S_\mathbb D)(z,z) = \tfrac{|f^\prime(z)|}{1-|f(z)|^2}$, $z\in \Omega$. Thus if $f^\prime$ is not zero on $\Omega$, then the curvature inequality takes the form
$$ K(z,w)^2 \preceq f^*(\mathbb S_\mathbb D)(z,w)^{-2} \mathcal G_{K-1}(z,w), \,\, z,w \in \Omega,$$
where $f^*(\mathbb S_\mathbb D)(z,w)^{2}$ is the kernel  $\tfrac{f^\prime(z)\overline{f^\prime(w)}}{(1-f(z)\overline{f(w)})^2}$.
As in the case of the disc, in particular, evaluating this inequality at a fixed but arbitrary point $z\in \Omega$, we have 
\begin{align*}
\partial \bar{\partial} \log K (z,z) \geq \sup\Big \{\tfrac{|f^\prime(z)|^2}{ (1 - |f(z)|^2)^2}	:  f\in {\rm Rat}(\Omega), \|f\|_{\Omega, \infty} \leq 1\Big \}  = \mathbb S_\Omega(z,z)^2,
\end{align*}
where $\mathbb S_\Omega$ is the Szeg\"{o} kernel of the domain $\Omega$. 
This is the curvature inequality for contractive homomorphisms {\rm  (see \cite[Corollary 1.2']{GMCI}) and also \cite{Ramiz}}).

We now  show that an analogue of Proposition \ref{eqn:curvinqstrong} is also valid for spherical contractions in $B_1(\mathbb B^m)$, where  $\mathbb B^m$ is the $m$-dimensional unit 
ball in $\mathbb C^m$. Recall that a commuting $m$-tuple $T=(T_1,\ldots,T_m)$
of operators on $\mathcal H$ is said to be a row contraction if $\sum_{i=1}^m T_iT_i^*\leq I.$  Let $K:\mathbb B^m\times\mathbb B^m\to\mathbb C$ be a sesqui-analytic positive definite kernel. Assume that the commuting $m$-tuple $M=(M_1,\ldots,M_m)$ of multiplication by the coordinate functions on $(\mathcal H, K)$ is in $B_1(\mathbb B^m)$. We let $ B_m(z,w):=\frac{1}{1-\langle z, w\rangle}$, $z,w\in\mathbb B^m$, be the reproducing kernel of the Drury-Arveson space. By \cite[Corollary 2]{DMS}, $M$ is a row contraction if and only if $B^{-1}_{m}(z,w) K(z,w)$ is non-negative definite on $\mathbb B^m$.  Thus, if  $M$ on $(\mathcal H, K)$ is a row contraction in $B_1(\mathbb B^m)$, applying Corollary \ref{gaussiancurv} for  $ B_m^{-1}(z,w)K(z,w)$ we obtain the following inequality:
\begin{equation}\label{eqnrow} 
K^2(z,w)  B_m^{-2}(z,w)\Big (\!\!\Big(\,\partial_i\bar{\partial}_j\log  B_m(z,w)\Big)\!\!\Big)_{i,j=1}^m
\preceq  B_m^{-2}(z,w)\mathcal G_{K^{-1}}(z,w).
\end{equation}
As before, evaluating at a  fixed but arbitrary point $z$ in $\mathbb B^m$, we obtain \cite[Corollary 2.3]{infinitelydivisiblemetric}.  

We now prove that the Gaussian curvature $\mathcal G_{K^{-1}}$ is monotone.
\begin{proposition}\label{monotone}
Let $\Omega\subset \mathbb C^m$ be a  domain. Suppose that $K_1$ and $K_2$  are two scalar valued positive definite kernels on $\Omega$ satisfying 
$K_1\succeq K_2$. Then 
$$  \mathcal G_{K_1^{-1}}(z,w) \succeq  \mathcal G_{K_2^{-1}}(z,w).
$$ 
\end{proposition}

\begin{proof}
Set $ K_3 = K_1-K_2$. By hypothesis,  $K_3$ is non-negative definite on $\Omega$. For $1\leq i,j\leq m$, a straightforward computation shows that 
\begin{align}\label{eqorderpres}
\begin{split}
K_1^2 \partial_i\bar{\partial}_j \log K_1 &= K_2^2 \partial_i\bar{\partial}_j \log K_2 + K_3^2 \partial_i\bar{\partial}_j \log K_3  \\
&\quad\quad \quad\quad + K_2\partial_i\bar{\partial}_jK_3 + K_3\partial_i\bar{\partial}_jK_2
   -\partial_iK_2\bar{\partial}_jK_3 - \partial_iK_3\bar{\partial}_jK_2.
\end{split}
\end{align}
Now set $\gamma_i(w)=  K_2(\cdot,w)\otimes \bar{\partial}_i K_3(\cdot,w)-\bar{\partial}_i K_2(\cdot,w)\otimes K_3(\cdot,w), 1\leq i\leq m$,
$w\in \Omega$. For $1\leq i,j\leq m$ and $z,w\in \Omega$, then we have 
\begin{align}\label{eqnmonotone}
\begin{split}
&\left\langle \gamma_j(w), \gamma_i(z)\right\rangle\\
&\quad\quad \quad = \big(K_2\partial_i\bar{\partial}_jK_3\big)(z,w)+ \big(K_3\partial_i\bar{\partial}_jK_2\big)(z,w)
   -\big(\partial_iK_2\bar{\partial}_j K_3\big)(z,w)-\big(\partial_iK_3\bar{\partial}_j K_2\big)(z,w).
\end{split}
\end{align}
Combining \eqref{eqorderpres} and \eqref{eqnmonotone}, we obtain 
\begin{align*}
& \big(\big(K_1^2 \partial_i\bar{\partial}_j \log K_1\big)(z,w) \big)_{i,j=1}^m\\
&\quad\quad\quad  = \big(\big(K_2^2\partial_i\bar{\partial}_j\log K_2\big)(z,w) \big)_{i,j=1}^m + 
\big(\big(K_3^2\partial_i\bar{\partial}_j\log K_3\big)(z,w)\big)_{i,j=1}^m + \big(\left\langle \gamma_j(w), \gamma_i(z)\right\rangle\big)_{i,j=1}^m. 
\end{align*}
Note that   $(z,w)\mapsto \big( \langle \gamma_j(w), \gamma_i(z)\rangle\big)_{i,j=1}^m$ is a non-negative definite kernel on $\Omega$ (see \cite[Lemma 2.1]{GM}). The proof is now complete since sum of two non-negative definite kernels remains non-negative definite.
\end{proof}
As a consequence of Proposition \ref{monotone}, we obtain the following  inequality for row contractions involving the Gaussian curvature.
\begin{corollary}\label{corGaussian}
Let $K:\mathbb B^m\times \mathbb B^m\to \mathbb C$ be a sesqui-analytic positive definite kernel. Assume that $K$ is normalized at the origin, that is, $K(z,0)=1,~z\in\mathbb B^m$. Suppose that the commuting tuple
$M$ of multiplication by the coordinate functions is a row contraction on $(\mathcal H, K).$ Then
\begin{equation}\label{weakerineq}
    \mathcal G_{K^{-1}}(z,w) \succeq \mathcal G_{B_m^{-1}}(z,w).
\end{equation}
\end{corollary}
\begin{proof} Since the tuple $M$ on $(\mathcal H, K)$ is a row contraction,  $\tilde{K}(z,w):=B_m^{-1}(z,w)K(z,w)$ defines a non-negative definite kernel on $\mathbb B^m.$ The kernel $\tilde{K}$ is normalized at 0 since $K$ is normalized at 0. Thus  $1=\tilde{K}(\cdot, 0)\in (\mathcal H, \tilde{K})$ and 
$$\|1\|^2_{(\mathcal H, \tilde{K})}=\langle \tilde{K}(\cdot, 0), \tilde{K}(\cdot, 0)\rangle_{(\mathcal H, \tilde{K})}=\tilde{K}(0,0)=1.$$
Hence it follows from \cite[Theorem 3.11]{PaulsenRaghupati} that $\tilde{K}\succeq 1$. Since the product of two non-negative definite kernels remain non-negative definite, multiplying both sides with $ B_m$, we get $K\succeq B_m.$ The proof is now complete by applying Proposition \ref{monotone}. 
\end{proof}

\begin{remark}
We point out that Corollary \ref{corGaussian}  can also be derived from \eqref{eqnrow}. In particular, in case $m=1$, Corollary \ref{corGaussian}
is a consequence of Proposition \ref{eqn:curvinqstrong}. But since the kernel $K_0$ satisfies the inequality $K_0\succeq \mathbb S_{\D}$,  it follows from Theorem \ref{k^2curvineq} that $\mathcal G_{K_0^{-1}}(z,w) \succeq \mathcal G_{\mathbb S_{\D}^{-1}}(z,w)$. Therefore,  the inequality \eqref{weakerineq}
is weaker than  the inequality \eqref{eqn:4} in case $m=1$.
\end{remark}

After establishing a lower bound for the Gaussian curvature of a non-negative definite kernel, we show that the partial derivatives are bounded from 
$(\hl ,K)$ to $(\hl, \mathcal G_{K^{-1}})$. We recall from \cite[Lemma   3.1]{infinitelydivisiblemetric} that $\Big (\!\!\Big(\,\partial_i\bar{\partial}_j K(z,w)\,\Big )\!\!\Big)_{i,j=1}^m$ is an non-negative definite kernel whenever $K$ is non-negative definite. 

\begin{theorem}\label{k^2curvineq}
Let $\Omega\subset \mathbb C^m$ be a  domain. Let $K:\Omega \times \Omega \to \mathbb C$ be a non-negative definite kernel. Suppose that the Hilbert space $(\mathcal{H}, K)$ contains the constant function $1$. Then 
$$\Big (\!\!\Big(\,\partial_i\bar{\partial}_j K(z,w)\!\Big )\!\!\Big)_{i,j=1}^m\preceq c\,\, \mathcal G_{K^{-1}}(z,w),$$
where $ c = \|1\|^2_{(\mathcal H,K)}$.
\end{theorem}

\begin{proof} 
Set $c=\|1\|^2_{(\mathcal{H}, K)}$. Choose an orthonormal basis $\{e_n(z)\}_{n\geq 0}$ of $(\mathcal H, K)$ with $e_0(z)=\frac{1}{\sqrt{c}}$. Then
$$K(z,w)-\frac{1}{c}=\sum_{i=1}^\infty
e_i(z)\overbar{e_i(w)}, ~z,w\in \Omega .$$
Hence $K(z,w)-\frac{1}{c}$ is non-negative definite on $\Omega\times\Omega$, or equivalently $cK-1$ is non-negative definite on $\Omega\times\Omega$. Therefore, by Corollary $\ref{cork^2curv}$, it follows that $\Big (\!\!\Big(\;(cK-1)^2\partial_i\bar{\partial}_j \log (cK-1)\;\Big )\!\!\Big)_{i,j=1}^m$ is non-negative definite on $\Omega\times\Omega$. Note that, for $z,w\in \Omega$, we have
\begin{align*}\big((cK-1)^2\partial_i&\bar{\partial}_j \log (cK-1)\big)(z,w)\\
=&(cK-1)(z,w)\big(\partial_i\bar{\partial}_j(cK-1)\big)(z,w)
-\big(\partial_i(cK-1)\big)(z,w)\big(\bar{\partial}_j(cK-1)\big)(z,w)\\ 
= & c^2K (z,w)\partial_i\bar{\partial}_j K(z,w)-
c\partial_i\bar{\partial}_j K(z,w)-
c^2\partial_iK(z,w )\bar{\partial}_jK(z,w)\\ 
= & c^2 K^2\partial_i\bar{\partial}_j \log K(z,w) 
-c\partial_i\bar{\partial}_j K(z,w).
\end{align*}
Hence we conclude that
$$\Big (\!\!\Big(\;\partial_i\bar{\partial}_j K(z,w)\;\Big )\!\!\Big)_{i,j=1}^m\preceq c\,\,\mathcal G_{K^{-1}}(z,w).$$
\vskip -1em
\end{proof}   

\begin{corollary}
Let $\Omega\subset \mathbb C^m$ be a  domain. Let $K:\Omega \times \Omega \to \mathbb C $ be a non-negative  
definite kernel. 
Then the linear operator $\boldsymbol \partial: (\mathcal{H}, K )\to  (\mathcal{H}, \mathcal G_{K^{-1}})$, where $\boldsymbol \partial f = (\partial_1 f, \ldots , \partial_m f)^{\rm tr}$, $f\in (\mathcal{H}, K )$, is bounded with $\|\boldsymbol \partial \| \leq \|1\|_{(\hl, K)}$.  
Moreover if $K$ is normalized at the point $w_0\in \Omega$, that is, $K(\cdot,w_0)$ is the constant function 1, then 
the linear operator $\boldsymbol \partial: (\mathcal{H}, K )\to  (\mathcal{H}, \mathcal G_{K^{-1}})$ is contractive.
\end{corollary}
\begin{proof}
To prove the first assertion of the corollary, note that the map $\boldsymbol \partial$ is unitary from $\ker \boldsymbol \partial^\perp$ to $(\mathcal{H},(\partial_i\bar{\partial}_j K)_{i,j=1}^m)$, and therefore is contractive from $(\hl,K)$ to $(\mathcal{H},(\partial_i\bar{\partial}_j K)_{i,j=1}^m)$.  
To complete the proof, it is therefore enough to show that $(\mathcal{H},(\partial_i\bar{\partial}_j K)_{i,j=1}^m)$ 
is contained in  $(\mathcal{H}, \mathcal G_{K^{-1}})$ 
and the inclusion map is bounded by $\|1\|_{(\mathcal H, K)}$.
This follows from Theorem \ref{k^2curvineq} using \cite[Theorem 6.25]{PaulsenRaghupati}.
For the second assertion, note that $\|1\|^2_{(\mathcal{H},K)}
=\left\langle K(\cdot,w_0),K(\cdot,w_0)\right \rangle_{(\mathcal{H}, K)}=K(w_0,w_0)=1$
by hypothesis and use Theorem  $\ref{k^2curvineq}$ to complete the proof.
%
\end{proof}

\section{A limit Computation}
Let $\Omega$ be a bounded domain in $\mathbb C^m$. Let $\boldsymbol M^* \in B_1(\Omega^*)$ be the  adjoint of  the $m$-tuple $\boldsymbol M$ of multiplication by the coordinate functions on a reproducing kernel Hilbert space $(\mathcal H, K)$ consiting of holomorphic functions on $\Omega\subset \mathbb C^m$.
Let  $\mathcal A(\Omega)$ be the function algebra of all those functions holomorphic in some open neighbourhood of the compact set  $\overbar{\Omega}$ equipped with the supremum norm on $\overbar{\Omega}$. The map $\mathbf m_f: h \mapsto f \cdot h$, $f\in \mathcal A(\Omega)$, $h\in (\mathcal H, K)$, where $(f\cdot h)(z) = f(z) h(z)$, defines a module multiplication for $(\mathcal H, K)$ over the algebra $\mathcal A(\Omega)$.  We let $\mathcal M:=(\mathcal H, K)$ denote this Hilbert module. Let $\mathcal M_0 \subseteq \mathcal M$ be a submodule. We now have a short exact sequence of Hilbert modules 
$$\begin{tikzcd}
0\arrow{r} &\mathcal M_0 
\arrow{r} {i} & {\mathcal M}
\arrow{r}{\pi} & \mathcal Q
\arrow{r}& 0
\end{tikzcd}$$
where $i$ is the inclusion map and $\pi$ is the quotient map. The problem of finding invariants for $\mathcal Q$ given the inclusion $\mathcal M_0 \subset \mathcal M$ has been studied in several papers (cf. \cite{DMQ, DMV}). A variant of this problem occurs by replacing the inclusion map with some other module map, for instance, one might set $\mathcal M_0 =  \varphi \mathcal M$ for some $\varphi \in \mathcal A(\Omega)$. Here we are going to consider the case of submodules $\mathcal M_0$ consisting of the maximal set of functions in $\mathcal M$ vanishing on some fixed subset $\mathcal Z$ of $\Omega$. A description of the specific examples we consider here follows.

 Let $K_1$ and $K_2$ be two scalar valued non-negative definite kernels  on $\Omega$. Assume that both the kernels sesqui-analtic. It is well known that $(\hl, K_1)\otimes (\hl, K_2)$ is  the reproducing kernel Hilbert space determined by the non-negative definite kernel $K_1\otimes K_2$, where $K_1\otimes K_2:(\Omega\times \Omega)\times(\Omega\times \Omega)\to \mathbb C$ is given by 
$$(K_1\otimes K_2)(z,\zeta;w,\rho)=K_1(z,w)K_2(\zeta,\rho),\; \; z,\zeta,w,\rho\in \Omega.$$  
We assume that the  operator $M_{z_i}$ of multiplication by the coordinate function $z_i$ is bounded on $(\mathcal H, K_1)$ as well as on $(\mathcal H, K_2)$ for  $i=1,\ldots,m$.  
Then $(\mathcal H, K_1) \otimes (\mathcal H, K_2)$ 
may be realized as a Hilbert module over the polynomial ring $\mathbb C[z_1,\ldots,z_{2m}]$ with the module action
defined by  
$$\mathbf m_p(h) = p h,\: h\in (\mathcal H, K_1) \otimes (\mathcal H, K_2),\: p\in \mathbb C[z_1,\ldots,z_{2m}].$$
The Hilbert space $(\mathcal H, K_1) \otimes (\mathcal H, K_2)$ admits a natural direct sum decomposition as follows.


For a non-negative integer $k$, 
let $\mathcal A_k$ be the subspace of 
$(\mathcal H ,K_1)\otimes (\mathcal H, K_2)$ defined by
\begin{equation}\label{eqA_k}
\mathcal A_k:=\big\{f\in (\mathcal H ,K_1)\otimes (\mathcal H ,K_2)
:\big(\big(\tfrac{\partial}{\partial \zeta}\big)^{\boldsymbol i} f(z,\zeta)\big)_{|\Delta}=0,
\;|\boldsymbol i|\leq k\big\},
\end{equation}
where $\boldsymbol i=(i_1,\ldots,i_m)\in {\mathbb Z}_+^m$, $|\boldsymbol i|=i_1+\cdots+i_m$, 
$\big(\tfrac{\partial}{\partial \zeta}\big)^{\boldsymbol i}=\frac{\partial^{|\boldsymbol i|}}{\partial \zeta_1^{i_1}\cdots\partial \zeta_m^{i_m}}$, and 
$\big(\big(\tfrac{\partial}{\partial \zeta}\big)^{\boldsymbol i} f(z,\zeta)\big)_{|\Delta}$ is 
the restriction of $\big(\tfrac{\partial}{\partial \zeta}\big)^{\boldsymbol i} f(z,\zeta)$
to the diagonal set $\Delta:=\{(z,z):z\in \Omega\}$.
It is easily verified that each of the subspaces $\mathcal A_k$ is closed and invariant under multiplication by any polynomial in $\mathbb C[z_1,\ldots, z_{2m}]$ and therefore they are sub-modules of $(\mathcal H ,K_1)\otimes (\mathcal H, K_2)$.
Setting $\mathcal S_0= \mathcal A_0^\perp$, $\mathcal S_k:= \mathcal A_{k-1} \ominus \mathcal A_{k}$, $k = 1, 2, \ldots$, we obtain  a direct sum decomposition of the Hilbert space $(\mathcal H ,K_1)\otimes (\mathcal H, K_2)$ as follows
$$
(\mathcal H ,K_1)\otimes (\mathcal H, K_2) = \bigoplus _{k=0}^\infty \mathcal S_k.$$
Define a linear map $\mathcal R_1:\hla \to {\rm Hol}(\Omega,\mathbb{C}^m)$ by setting 
\begin{equation}\label{the map R_1}
\mathcal R_1(f)=
\frac{1}{\sqrt{\alpha\beta(\alpha+\beta)}}\begin{pmatrix}
(\beta\partial_{1}f-\alpha\partial_{m+1}f)_{|\Delta}\\
\vdots\\
(\beta\partial_{m}f-\alpha\partial_{2m}f)_{|\Delta}
\end{pmatrix}
\end{equation}
for $f\in \hla$.
Let  $\iota: \Omega \to \Omega\times \Omega$ be the map $\iota(z) = (z,z)$,  $z\in\Omega$. Any Hilbert module $\mathcal  M$ over the polynomial ring $\mathbb C[z_1,\ldots,z_{m}]$ may be thought of as a module $\iota_\star\mathcal M$ over the ring $\mathbb C[z_1,\ldots,z_{2m}]$ by re-defining the multiplication: $\mathbf m_p(h) = (p \circ \iota) h$, $h\in \mathcal M$ and $p \in  \mathbb C[z_1,\ldots,z_{2m}]$.
The module $\iota_\star\mathcal M$ over $\mathbb C[z_1,\ldots,z_{2m}]$ is defined to be the push-forward of the module $\mathcal M$ over $\mathbb C[z_1,\ldots,z_m]$  under the inclusion map $\iota$.

\begin{theorem}\label{module on 2nd box}\rm (\cite[Theorem 3.5.]{GM})
Suppose $K:\Omega\times\Omega\to \mathbb C$ is a sesqui-analytic function such that the functions $K^{\alpha}$ and $K^{\beta}$, defined on $\Omega\times\Omega$, are non-negative definite for some $\alpha,\beta>0$. Then the followings  hold:
\begin{enumerate}
\item[(1)] $\ker \mathcal R_1=\mathcal S_1^\perp$ and  $\mathcal R_1$ maps  $\mathcal S_1$ isometrically onto 
$\big(\hl, \mathbb{K}^{(\alpha,\beta)}\big)$.
\item[(2)] 
Suppose that the  operator $M_i$ of multiplication by the co-ordinate function $z_i$ is 
 bounded on both $(\hl, K^\alpha)$ and $(\hl, K^{\beta})$ for $i=1,2,\ldots,m$ .
Then the Hilbert module $\mathcal S_1$ is isomorphic to the push-forward module $\iota_\star \big(\hl, \mathbb{K}^{(\alpha,\beta)}\big)$  via the module map ${\mathcal R_1}_{|\mathcal S_1}$.
\end{enumerate}
\end{theorem}

We consider the example of the Hardy space.
Let $K_1(z,w)=K_2(z,w)=\frac{1}{1-z\bar{w}}$, $z,w\in \mathbb D,$ be the  Szeg\"o kernel of the  unit disc $\D$. In this case $(\mathcal H, K_1)\otimes (\mathcal H, K_2)$ is the Hardy space on the bidisc $\mathbb D^2$, and it is often denoted by $H^2(\mathbb D^2)$. Now, we can compute the kernel functions for $\mathcal S_0$ and $\mathcal A_0$ in this example as follows, see \cite{DM}. 
The vectors  $\Big\{\frac{e_k}{\sqrt{k+1}}\Big\}_{k\geq 0}$ form an orthonormal basis of $\mathcal S_0$, where $e_k$ is given by
$$e_k(z_1,z_2)=\sum_{j=0}^k z_1^jz_2^{k-j}, z_1,z_2\in \mathbb D.$$
Therefore the reproducing kernel $K_{\mathcal S_0}$ of $\mathcal S_0$ is given by
$$K_{\mathcal S_0}(\boldsymbol z,\boldsymbol w)=\sum_{k\geq 0}\frac{e_k(\boldsymbol z)\overline{e_k(\boldsymbol w)}}{k+1}~\boldsymbol z=(z_1,z_2), \boldsymbol w=(w_1,w_2)\in \mathbb D^2.$$
A cloosed forf expression for $K_{\mathcal S_0}$ is easily obtained: 
$$K_{\mathcal S_0}(\boldsymbol z,\boldsymbol z)=\frac{1}{|z_1-z_2|^2}\log \frac{|1-z_1\bar{z}_2|^2}{(1-|z_1|^2)(1-|z_2|^2)},~ \boldsymbol z=(z_1,z_2)\in \mathbb D^2.$$
Therefore it follows that
\begin{align*}
&~~~K_{\mathcal A_0}(\boldsymbol z, \boldsymbol z)\\&=\frac{1}{(1-|z_1|^2)(1-|z_2|^2)}-K_{\mathcal S_0}(\boldsymbol z,\boldsymbol z)\\
&=\frac{1}{(1-|z_1|^2)(1-|z_2|^2)}-\frac{1}{|z_1-z_2|^2}\log \frac{|1-z_1\bar{z}_2|^2}{(1-|z_1|^2)(1-|z_2|^2)}\\
&=\frac{1}{(1-|z_1|^2)(1-|z_2|^2)}-\frac{1}{|z_1-z_2|^2}\Big(\frac{|z_1-z_2|^2}{(1-|z_1|^2)(1-|z_2|^2)}-\frac{1}{2}\frac{|z_1-z_2|^4}{(1-|z_1|^2)^2(1-|z_2|^2)^2}+\cdots\Big).
\end{align*}
We are now in a position to find the kernel function for the  module $\mathcal S_1$, which is nothing but the limit:  
\begin{equation}
\lim_{z_2\to z_1} \frac{K_{\mathcal A_0}(\boldsymbol z, \boldsymbol z)}{|z_1-z_2|^2}=\frac{1}{2} \frac{1}{(1-|z_1|^2)^4}, (z_1,z_2)\in \mathbb D^2.
\end{equation}
Consider the short exact sequence $0\to \mathcal A_0\to H^2(\mathbb D^2) \to \mathcal S_0\to 0$. It is known that the quotient module $\mathcal S_0$ is the pushforward of the Bergman module on the disc. We note that  
$$\mathcal K_{\mathcal S_1}(\boldsymbol z) = \mathcal K_{\mathcal S_0}(\boldsymbol z)+\frac{2}{(1-|z_1|^2)^2}, \,\, \boldsymbol z \in \Delta =\{(z,z):z\in \mathbb D\}.$$
Thus  the restriction of $\mathcal K_{\mathcal A_0}$ to the zero set $\Delta$ might serve as an invariant for the inclusion $\mathcal M_0 \subset \mathcal M$. This possibility is explored below in a class of examples. 


Let $\Omega \subset \mathbb{C}^m$ be a bounded domain and $K:\Omega\times\Omega\to \mathbb C$ be a sesqui-analytic function such that the functions $K^{\alpha}$ and $K^{\beta}$ are non-negative definite on $\Omega\times\Omega$ for some $\alpha,\beta>0$. For a non-negative integer $p$, let $K_{{\mathcal A}_p}$ be the reproducing kernel of 
${\mathcal A}_p$, where ${\mathcal A}_p$ is defined in $\eqref{eqA_k}$.

To prove the main result of this section, we need the following two lemmas. One way to prove both of the  lemmas is to make the change of variables 
$$u_1(z,\zeta)= \tfrac{1}{2}( z_1- \zeta_1), \ldots, u_m(z,\zeta)=\tfrac{1}{2}(z_m -\zeta_m);\,\, v_1(z,\zeta) = \tfrac{1}{2}( z_1+\zeta_1), \ldots , v_m(z,\zeta)=\tfrac{1}{2}(z_m + \zeta_m).$$ We give the details for the proof of the first lemma. The proof for the second one follows by similar arguments. 

\begin{lemma}\label{lemlimit1}
Let $\Omega\subset \mathbb C^m$ be a domain and let $\Delta$ be the diagonal set $\{(z,z):z\in \Omega\}$. Suppose that $f:\Omega\times \Omega\to \mathbb C$ is a holomorphic function satisfying $f_{|\Delta}=0.$ Then for each $z_0\in \Omega$, there exists a neighbourhood $\Omega_0\subset \Omega$ (independent of $f$) of $z_0$  and holomorphic functions $f_1,f_2,\ldots,f_m$ on $\Omega_0\times \Omega_0$ such that
$$f(z,\zeta)=\sum_{i=1}^{m}(z_i-\zeta_i)f_i(z,\zeta),\; z=(z_1,\ldots,z_m),\zeta=(\zeta_1,\ldots,\zeta_m)\in  \Omega_0.$$  
\end{lemma}
\begin{proof}
Note that the image of the diagonal set $\Delta\subseteq \Omega\times \Omega$ under the map $\varphi:\Omega\times \Omega \to \mathbb C^{2m}$, where $$\varphi(z,\zeta):=(u_1(z,\zeta),\ldots,u_m(z,\zeta),v_1(z,\zeta),\ldots,v_m(z,\zeta)),$$ is the set $\{(0,w): w\in \Omega\}$.
Therefore we may choose a neighbourhood of $(0,z_0)$ which is a polydisc contained in $\widehat{\Omega}:=\varphi(\Omega \times \Omega)$. Suppose $f$ is a holomorphic function on $\Omega\times \Omega$ vanishing on the set $\Delta$. Setting $g:=f\circ\varphi^{-1}$ on $\hat{\Omega}$, we see that $g$ is a holomorphic function on $\hat{\Omega}$ vanishing on the set $\{(0,w): w\in \Omega\}$. Therefore $g$
has a power series representation around $(0,z_0)$ 
of the form 
$\sum_{\boldsymbol i,\boldsymbol j\in\mathbb Z^m_+} a_{\boldsymbol i \boldsymbol j} { u}^{\boldsymbol i} ( v- z_0)^{\boldsymbol j},$ where $\sum_{\boldsymbol j\in \mathbb Z^m_+} a_{0 \boldsymbol j} ( v- z_0)^{\boldsymbol j}= 0$  on the chosen polydisc. Hence $a_{0 \boldsymbol j}=0$ for all $\boldsymbol j \in \mathbb Z^m_+$, and the power series of $g$ is of  the form $\sum_{\ell=1}^m u_\ell g_\ell( u, v)$, where $$g_\ell( u, v) = \sum_{\boldsymbol i \boldsymbol j} a_{\boldsymbol i \boldsymbol j}u^{\boldsymbol {i}-e_\ell}( v- z_0)^{\boldsymbol {j}},\;1\leq \ell\leq m.$$ 
Here the sum is over all multi-indices $\boldsymbol {i}=(i_1,\ldots,i_m)$ satisfying $i_1=0, \ldots, i_{\ell-1}=0, i_\ell \geq 1$ while  $\boldsymbol {j}$ remains arbitrary. Pulling this expression back to $\Omega\times \Omega$ under the bi-holomorphic map $\varphi$, we obtain the expansion of $f$ in a neighbourhood of $(z_0,z_0)$ as prescribed in the Lemma \ref{lemlimit1}. 
\end{proof}

\begin{lemma}\label{lemlimit2}
Suppose that $f:\Omega\times \Omega\to \mathbb C$ is a holomorphic function satisfying $f_{|\Delta}=0$ and $\big(\big(\frac{\partial}{\partial \zeta_j}\big)f(z,\zeta)\big)_{|\Delta}=0$, $j=1,\ldots,m$. Then for each $z_0\in \Omega$, there exists a neighbourhood $\Omega_0\subset \Omega$ (independent of $f$) of $z_0$ and holomorphic functions $f_{ij}$,  $1\leq i\leq j\leq m$, on $\Omega_0\times \Omega_0$ such that
 $$f(z,\zeta)=\sum_{1\leq i\leq j\leq m}(z_i-\zeta_i)(z_j-\zeta_j)f_{ij}(z,\zeta), \; z,\zeta\in  \Omega_0.$$  
\end{lemma}

\begin{theorem} \label{limitq}
Let $\Omega \subset \mathbb{C}^m$ be a bounded domain and $K:\Omega\times\Omega\to \mathbb C$ be a sesqui-analytic function such that the functions $K^{\alpha}$ and $K^{\beta}$ are non-negative definite on $\Omega\times\Omega$ for some $\alpha,\beta>0$.
For $z$ in $\Omega$ and $1\leq i,j \leq m$, we have 
$$\lim_{\substack{{\zeta_i\to z_i}\\{\zeta_j\to z_j}}}
\left(\frac{K_{\mathcal A_0}(z,\zeta;z,\zeta)}{(z_i-\zeta_i)(\bar{z}_j-\bar{\zeta}_j)}\bigg|_{\zeta_l=z_l,l \neq i,j}\right)  
=\tfrac{\alpha\beta}{(\alpha+\beta)} K(z,z)^{\alpha+\beta}\partial_i\bar{\partial}_j\log K(z,z),$$
where $K_{\mathcal A_0}$ is the reproducing kernel of the subspace $\mathcal A_0$, and
$\tfrac{K_{\mathcal A_0}(z,\zeta;z,\zeta)}{(z_i-\zeta_i)(\bar{z}_j-\bar{\zeta}_j)}\bigg|_{\zeta_l=z_l,l \neq i,j}$ is the restriction of the function $\tfrac{K_{\mathcal A_0}(z,\zeta;z,\zeta)}{(z_i-\zeta_i)(\bar{z}_j-\bar{\zeta}_j)}$ to the set $\big\{(z,\zeta)\in \Omega\times \Omega:z_l=\zeta_l,l=1,\ldots,m, l\neq i,j\big\}$.

\end{theorem}

\begin{proof}
Let $K_{{\mathcal A}_0\ominus {\mathcal A}_1}(z,\zeta;w,\nu)$ be the reproducing kernels 
of ${\mathcal A}_0\ominus {\mathcal A}_1.$ Fix a point $z_0$ in $\Omega$. Choose a neighbourhood $\Omega_0$ of $z_0$ in $\Omega$ such that the conclusions of Lemma \ref{lemlimit1} and Lemma \ref{lemlimit2} are valid. Now we restrict the kernels $K^{\alpha}$ and  $K^{\beta}$ to $\Omega_0\times \Omega_0.$  

Let $f$ be an arbitrary function in $\mathcal A_1$. Then, by definition, $f$ satisfies the hypothesis of Lemma \ref{lemlimit2}, and therefore, it follows that  
\begin{equation}\label{eqnlimit5}
\lim_{\zeta_i\to z_i}\left(\frac{f(z,\zeta)}{(z_i-\zeta_i)}\bigg|_{z_l=\zeta_l,l \neq i}\right) = 0,\;\;i=1,\ldots,m.
\end{equation}
Let $\{h_n\}_{n\in\z}$ be an orthonormal basis of $\mathcal A_1$. Since the series $\sum_{n=0}^\infty h_n(z,\zeta)\overbar{h_n(z,\zeta)}$ converges uniformly to $K_{\mathcal A_1}(z,\zeta;z,\zeta)$ on the compact subsets of 
$\Omega_0\times \Omega_0$, using \eqref{eqnlimit5} we see that
\begin{align*}
\lim_{\substack{{\zeta_i\to z_i}\\{\zeta_j\to z_j}}}
\left(\frac{K_{\mathcal A_1}(z,\zeta;z,\zeta)}{(z_i-\zeta_i)(\bar{z}_j-\bar{\zeta}_j)}\bigg|_{\zeta_l=z_l,l \neq i,j}\right)&= 
\sum_{n=0}^\infty\lim_{\zeta_i\to z_i}\left(\frac{h_n(z,\zeta)}{(z_i-\zeta_i)}\bigg|_{z_l=\zeta_l,l \neq i}\right) \lim_{\zeta_j\to z_j}\overbar{\left(\frac{h_n(z,\zeta)}{(z_j-\zeta_j)}\bigg|_{z_l=\zeta_l,l \neq j}\right)}\\
&=0.
\end{align*}
Since 
$K_{\mathcal A_0}=K_{\mathcal A_0\ominus \mathcal A_1}+K_{\mathcal A_1}$, the above equality leads to 
$$\lim_{\substack{{\zeta_i\to z_i}\\{\zeta_j\to z_j}}}
\left(\frac{K_{\mathcal A_0}(z,\zeta;z,\zeta)}{(z_i-\zeta_i)(\bar{z}_j-\bar{\zeta}_j)}\bigg|_{\zeta_l=z_l,l \neq i,j}\right) 
=\lim_{\substack{{\zeta_i\to z_i}\\{\zeta_j\to z_j}}}
\left(\frac{K_{\mathcal A_0\ominus \mathcal A_1}(z,\zeta;z,\zeta)}{(z_i-\zeta_i)(\bar{z}_j-\bar{\zeta}_j)}\bigg|_{\zeta_l=z_l,l \neq i,j}\right).$$

Now let $\{e_n\}_{n\in\z}$ be an orthonormal basis of ${\mathcal A}_0\ominus {\mathcal A}_1$.
Since each $e_n\in {\mathcal A}_0$, by Lemma \ref{lemlimit1}, there exist holomorphic functions 
${e_{n, i}}$, $1\leq i\leq m$, on $\Omega_0\times \Omega_0$ such that 
$$e_n(z,\zeta)=\sum_{i=1}^m (z_i-\zeta_i)e_{n, i}(z,\zeta),\;z,\zeta\in \Omega_0.$$
Thus for $1\leq i\leq m$, we have
\begin{equation}\label{restriction}
\lim_{\zeta_i\to z_i}\left(\frac{e_n(z,\zeta)}{(z_i-\zeta_i)}\bigg|_{z_l=\zeta_l,l \neq i}\right) = e_{n, i}(z,z),\;z\in \Omega_0.
\end{equation}
Since the series $\sum_{n=0}^{\infty} e_n(z,\zeta)\overbar{e_n(z,\zeta)}$ converges to $K_{\mathcal A_0\ominus \mathcal A_1}$
uniformly on compact subsets of $\Omega_0\times \Omega_0$,
using \eqref{restriction}, we see that
\begin{equation}\label{eqnlimit1}
\lim_{\substack{{\zeta_i\to z_i}\\{\zeta_j\to z_j}}}
\left(\frac{K_{\mathcal A_0\ominus \mathcal A_1}(z,\zeta;z,\zeta)}{(z_i-\zeta_i)(\bar{z}_j-\bar{\zeta}_j)}\bigg|_{\zeta_l=z_l,l \neq i,j}\right)=\sum_{n=0}^\infty e_{n, i}(z,z)\overbar {e_{n, j}(z,z)},\;z\in \Omega_0.
\end{equation}
Recall that by Theorem \ref{module on 2nd box}, the map ${\mathcal R}_1:{\mathcal A}_0\ominus {\mathcal A}_1\to (\mathcal{H}, \mathbb K^{(\alpha,\beta)})$ given by
$$ {\mathcal R}_1 f=\frac{1}{\sqrt{\alpha\beta(\alpha+\beta)}}\begin{pmatrix}
(\beta\partial_{1} f-\alpha\partial_{m+1} f)_{|\Delta}\\
\vdots\\
(\beta\partial_{m} f-\alpha\partial_{2m} f)_{|\Delta}
\end{pmatrix},\; f\in \mathcal A_0\ominus \mathcal A_1$$
is unitary. Hence $\{{\mathcal R}_1(e_n)\}_n$ is an orthonormal basis for $(\mathcal{H},\mathbb K^{(\alpha,\beta)})$ and consequently 
\begin{equation}\label{onb sum}
\sum_{n=0}^\infty{\mathcal R}_1(e_n)(z){{\mathcal R}_1(e_n)(w)}^*=\mathbb K^{(\alpha,\beta)}(z,w),\;z,w\in \Omega_0.
\end{equation}
A direct computation shows that 
$$\big((\beta\partial_i-\alpha\partial_{m+i})e_n(z,\zeta)\big)_{|\Delta}=(\alpha+\beta)e_{n, i}(z,\zeta)_{|\Delta},\,\, 1\leq i\leq m,\; n\geq 0.$$
Therefore ${\mathcal R}_1(e_n)(z)=\sqrt{\frac{\alpha+\beta}{\alpha\beta}}
\begin{pmatrix}
e_{n, 1}(z,z)\\
\vdots\\
e_{n, m}(z,z)
\end{pmatrix}.$ Thus using $(\ref{onb sum})$ we obtain
$$\sum_{n=0}^{\infty} \begin{pmatrix}
e_{n, 1}(z,z)\\
\vdots\\
e_{n, m}(z,z)
\end{pmatrix} 
\begin{pmatrix}
e_{n, 1}(z,z)\\
\vdots\\
e_{n, m}(z,z)
\end{pmatrix}^{*}
=\tfrac{\alpha\beta}{(\alpha+\beta)}\mathbb K^{(\alpha,\beta)}(z,z),\;\;z\in \Omega_0.$$
Now the proof is complete using \eqref{eqnlimit1}.
\end{proof}
\begin{remark}
Let $H(z)= \big (\!\!\big (\langle s_i(z), s_j(z)\rangle \big )\!\!\big )_{i,j=1}^n$, $z\in \Omega$, be the Hermitian metric of a holomorphic (trivial) vector bundle $E$ defined on $\Omega$ relative to the holomorpohic frame $\{s_1, \ldots ,s_n\}$. The curvature $\mathbf K_H$ of the vector bundle $E$ is the $(1,1)$ form 
$$\sum_{i,j=1}^n\overbar{\partial}_j\big ( H^{-1}\partial _i H\big ) d\bar{z}_j \wedge d z_i.$$
The trace of the curvature $\mathbf K_H$ is obtained by replacing each of the coefficients $\overbar{\partial}_j\big ( H^{-1}\partial _i H\big )$ by their trace. 
Recall that the determinant bundle $\det E$ is a line bundle determined by the holomorphic frame $s_1\wedge \cdots \wedge s_n$ and the Hermitian metric: $H(z):=\det H(z)$.
The trace of the curvature of the vector bundle $E$ and the curvature of the determinant bundle $\det E$ are equal, i.e.,  $\text{\rm trace}(\mathbf K_H) = \mathbf K_{\det H}$, see \cite[Equation (4.6)]{JPD}. 

Now, from Theorem \ref{limitq}, we see that the Hermitian structure for the Hilbert module $\mathcal S_1$ is $\mathbb K^{(\alpha, \beta)}$. Aslo, we have the following equality:
$$\text{\rm trace}\big (\mathbf K_{\mathbb K^{(\alpha,\beta)}} \big ) = \tfrac{1}{m}\mathbf K_{K^{\alpha + \beta}} + \mathbf K_{\det (\mathcal K_K)}.$$
Thus, in these examples, we see that $\text{\rm trace}\big (\mathbf K_{\mathbb K^{(\alpha,\beta)}} \big )$ is a function of $\alpha + \beta$.  Consequently, if $\alpha +\beta = \alpha^\prime +\beta^\prime$, then  we have 
$$
\begin{tikzcd}
0 \arrow{r} &\mathcal A_0  \arrow{r}{i}
& (\mathcal H, K^\alpha)\otimes (\mathcal H, K^\beta)\arrow{r}{\pi} & (\mathcal H, K^{\alpha+\beta}) \arrow[equal]{d} \arrow{r}& 0\\
0\arrow{r} &\mathcal A_0 \arrow{r}{i} & (\mathcal H, K^{\alpha^\prime})\otimes (\mathcal H, K^{\beta^\prime})\arrow{r}{\pi}& (\mathcal H, K^{\alpha^\prime +\beta^\prime}) \arrow{r}& 0,
\end{tikzcd}
$$
and  $\text{\rm trace}\big (\mathbf K_{\mathbb K^{(\alpha,\beta)}} \big ) = \text{\rm trace}\big (\mathbf K_{\mathbb K^{(\alpha^\prime,\beta^\prime)}} \big )$. 
Replacing the equality of the quotient modules $(\mathcal H, K^{\alpha+\beta})$ 
and $(\mathcal H, K^{\alpha^\prime +\beta^\prime})$ by an isomorphism does not change the conclusion. In general, replace $K^\alpha$ by $K_1$; $K^\beta$ by $K_2$, and $K^{\alpha^\prime}$ by $K^\prime_1$; $K^{\beta^\prime}$ by $K^\prime_2$ and assume that $K_1^\prime (w,w)  K^\prime_2(w,w) = \varphi(w) K_1(w,w) K_2(w,w) \overline{\varphi(w)}$ for some non-vanishing holomorphic function defined on an open subset $U\subset \Omega$. This means that the quotient modules $\mathcal S_0$ and $\mathcal S_0^\prime$ are equivalent. A straightforward computation then shows that $\text{\rm trace}\big (\mathbf K_{K_{12}} \big ) = \text{\rm trace}\big ( \mathbf K_{K^\prime_{12}} \big )$, 
where $K_{12} = \mathcal G_{(K_1 K_2)^{-1}}$ and similarly, $K^\prime_{12} = \mathcal G_{(K^\prime_1 K^\prime_2)^{-1}}$.
Hence $\text{\rm trace}\big (\mathbf K_{K_{12}} \big)$
is an invariant of the short exact sequences of the form 
$$\begin{tikzcd} 0 \arrow{r} &\mathcal A_0  \arrow{r}{i}
& (\mathcal H, K_1)\otimes (\mathcal H, K_2)\arrow{r}{\pi} & (\mathcal H, K_1 K_2) \arrow{r} \arrow{r}& 0. \end{tikzcd}$$
We  expect  this  to  be  the  case  in  much  greater  generality. 


 \end{remark}
The following corollary is immediate by choosing $\alpha = 1 = \beta$ in Theorem \ref{limitq}. It also gives an alternative for computing the Gaussian curvature defined in \eqref{Gauss} whenever the metric is of the form $K(z,z)^{-1}$ for some positive definite kernel $K$ defined on $\Omega\times \Omega$, where $\Omega\subset \mathbb C$ is a bounded domain. Indeed, the assumption that $\boldsymbol T$ is in $B_1(\Omega)$ is not necessary to arrive at the formula in the corollary below. 
\begin{corollary}
Let $\boldsymbol T$ be a commuting $m$-tuple in  the Cowen-Douglas class $B_1(\Omega)$ realized as the adjoint of the $m$-tuple $M$ of multiplication operators by coordinate functions on a reproducing kernel Hilbert space $(\mathcal H, K) \subseteq {\rm Hol}(\Omega_0),$ for some open subset $\Omega_0$ of $\Omega$.  Then the curvature $\mathcal K_{\boldsymbol T}(z)=\Big(\mathcal K_{\boldsymbol T}(z)_{i,j}\Big)_{i,j=1}^m$ is  given by the formula 
$$\mathcal K_{\boldsymbol T}(z)_{i,j} = \frac{2}{K(z,z)^2}\lim_{\substack{{\zeta_i\to z_i}\\{\zeta_j\to z_j}}}
\left(\frac{K_{\mathcal A_0}(z,\zeta;z,\zeta)}{(z_i-\zeta_i)(\bar{z}_j-\bar{\zeta}_j)}\bigg|_{\zeta_l=z_l,l \neq i,j}\right),\,\,  
 z\in \Omega, \,\, 1 \leq i,j \leq m.$$
\end{corollary}


\end{document}